\documentclass{amsart}   	
\usepackage{amssymb,amsthm}
\usepackage{amsthm}
\usepackage{yhmath}
\usepackage{graphicx}
\usepackage{enumitem}
\usepackage{tikz}
\usepackage{tikz-cd}
\usetikzlibrary{calc}
\usetikzlibrary{decorations.pathmorphing}
\usepackage{hyperref}
\usepackage{adjustbox}
\usepackage{mathtools}
\usepackage[all]{xy}

\tikzset{curve/.style={settings={#1},to path={(\tikztostart)
    .. controls ($(\tikztostart)!\pv{pos}!(\tikztotarget)!\pv{height}!270:(\tikztotarget)$)
    and ($(\tikztostart)!1-\pv{pos}!(\tikztotarget)!\pv{height}!270:(\tikztotarget)$)
    .. (\tikztotarget)\tikztonodes}},
    settings/.code={\tikzset{quiver/.cd,#1}
        \def\pv##1{\pgfkeysvalueof{/tikz/quiver/##1}}},
    quiver/.cd,pos/.initial=0.35,height/.initial=0}

\tikzset{tail reversed/.code={\pgfsetarrowsstart{tikzcd to}}}
\tikzset{2tail/.code={\pgfsetarrowsstart{Implies[reversed]}}}
\tikzset{2tail reversed/.code={\pgfsetarrowsstart{Implies}}}
\tikzset{no body/.style={/tikz/dash pattern=on 0 off 1mm}}

\graphicspath{ {./images/} }

\newtheorem{theorem}{Theorem}[section]
\newtheorem{proposition}[theorem]{Proposition}
\newtheorem{lemma}[theorem]{Lemma}

\theoremstyle{definition}

\newtheorem{remark}[theorem]{Remark}

\usepackage[utf8]{inputenc}

\newcommand{\uxa}{\ensuremath{(\underline{X},\underline{A})}}

\newcommand{\caa}{\ensuremath{(\underline{CA},\underline{A})}}

\newcommand{\zk}{\ensuremath{\mathcal{Z}_{K}}}

\newcounter{bean}

\newcommand{\namedright}[3]{\ensuremath{#1\stackrel{#2}
 {\longrightarrow}#3}}
\newcommand{\nameddright}[5]{\ensuremath{#1\stackrel{#2}
 {\longrightarrow}#3\stackrel{#4}{\longrightarrow}#5}}

\newcommand{\qqed}{\hfill\Box}

\title{Polyhedral products associated to pseudomanifolds}

\author{Lewis Stanton} 
\address{Mathematical Sciences, University of Southampton, Southampton SO17 1BJ, United Kingdom}
\email{lrs1g18@soton.ac.uk}

\author{Stephen Theriault}
\address{Mathematical Sciences, University of Southampton, Southampton SO17 1BJ, United Kingdom}
\email{s.d.theriault@soton.ac.uk}


\begin{document}

\begin{abstract}
We study the homotopy theory of polyhedral products associated to a combinatorial generalisation of manifolds known as pseudomanifolds. As special cases, we show that loop spaces of moment-angle manifolds associated to triangulations of $S^2$ and $S^3$ decompose as a product of spheres and loops on spheres.
\end{abstract}

\maketitle

\section{Introduction}
\label{sec:intro}

Polyhedral products are subspaces of the Cartesian product, the properties of which are governed by an underlying simplicial complex. They unify various constructions across mathematics, such as complements of complex coordinate subspace arrangements in combinatorics, intersections of quadrics in complex geometry and classifying spaces of graph products of groups in geometric group theory. Understanding their homotopy theory has implications in all these areas. In this paper, we study the homotopy theory of polyhedral products associated to a family of simplicial complexes known as pseudomanifolds.

Let $K$ be a simplicial complex on the vertex set $[m] = \{1,2,\cdots,m\}$. For $1 \leq i \leq m$, let $(X_i,A_i)$ be a pair of pointed $CW$-complexes, where $A_i$ is a pointed $CW$-subcomplex of $X_i$. Let $\uxa = \{(X_i,A_i)\}_{i=1}^m$ be the sequence of pairs. For each simplex $\sigma \in K$, let $\uxa^\sigma$ be defined by \[ \uxa^\sigma = \prod\limits_{i=1}^m Y_i \text{ where }  Y_i = \begin{cases} X_i & i \in \sigma \\ A_i & i \notin \sigma. \end{cases}\] The \textit{polyhedral product} determined by $\uxa$ and $K$ is \[\uxa^K = \bigcup\limits_{\sigma \in K} \uxa^{\sigma} \subseteq \prod\limits_{i=1}^m X_i.\] A particularly important special case is when $(X_i,A_i) = (D^2,S^1)$ for all $i$. These polyhedral products are called \textit{moment-angle complexes}, and are denoted $\mathcal{Z}_K$.

One aspect of the homotopy theory of polyhedral products that has been the subject of intense study recently is the homotopy type of their based loop spaces. Let $\mathcal{P}$ be the collection of $H$-spaces homotopy equivalent to a finite type product of spheres and loops on spheres. If $X$ is a simply-connected space, there are advantages to knowing that $\Omega X\in\mathcal{P}$: it means the homology of $\Omega X$ is torsion-free, the Steenrod algebra acts trivially on the mod-$p$ cohomology of $\Omega X$ for any prime $p$, and if the factors in the decomposition of $\Omega X$ are explicit, then the homotopy groups of~$X$ are known to the same extent as the homotopy groups of spheres. 
Various families of polyhedral products have been shown to have their loop space in $\mathcal{P}$, including flag complexes \cite{PT,V}, graphs~\cite{St1}, $2$-dimensional simplicial complexes with torsion free homology \cite{St2} and certain polyhedral join products \cite{E}. In this paper, we focus on the case when $K$ is the triangulation of a sphere.

If $K$ is the triangulation of a sphere, then $\zk$ has the structure of a manifold, and is known as a \emph{moment-angle-manifold}. The diffeomorphism type of $\zk$ is known for an important family of triangulations. If $P$ is a simple polytope obtained from a simplex by iterated vertex cuts and~$K$ is the Alexander dual of the boundary of $P$, then $\zk$ is diffeomorphic to a connected sum of products of two spheres. This statement originated in work of MacGavran~\cite{M} pre-dating polyhedral products, took a spectacular leap forward in work of Bosio and Meersseman~\cite{BM} and Gitler and L\'{o}pez de Medrano~\cite{GLdM} on intersections of quadrics, and culminated in the solution of a conjecture in~\cite{GLdM} by Chen, Fan and Wang~\cite{CFW}. However, very little is known about even the homotopy type of moment-angle manifolds for triangulations of spheres outside this family. 

In this paper we develop new methods to study the homotopy type of $\Omega\zk$ for a combinatorial generalisation of triangulations of spheres known as pseudomanifolds. The collection of pseudomanifolds include triangulations of manifolds. We establish criteria for when a polyhedral product of the form $\caa^{K}$ with $K$ a pseudomanifold has $\Omega\caa^{K}\in\mathcal{P}$. In particular, we prove the following. 

\begin{theorem}
\label{thm:introtriangsurface}
    If $K$ is the triangulation of a connected, orientable, closed surface on $[m]$, then $\Omega \zk \in \mathcal{P}$.
\end{theorem}

This includes the case when $K$ is a triangulation of $S^2$. We also obtain an analogous result when~$K$ is a triangulation of $S^3$.

\begin{theorem}
\label{thm:introtriofS3inP}
     Let $K$ be a triangulation of $S^3$ on $[m]$. Then $\Omega \zk \in \mathcal{P}$.
\end{theorem} 

The argument proving Theorem~\ref{thm:introtriofS3inP} breaks into two cases, one of which can be generalised to certain triangulations of any odd dimensional sphere. A simplicial complex is \emph{$k$-neighbourly} if every set of $k+1$ vertices spans a simplex. A triangulation $K$ of $S^{2n+1}$ is \emph{neighbourly} if $K$ is $n$-neighbourly. 

\begin{theorem}
\label{thm:introneighbourlytriinP}
     If $K$ is a neighbourly triangulation of $S^{2n+1}$ on $[m]$ with $n \geq 1$, then $\Omega \zk \in \mathcal{P}$.
\end{theorem}

The methods used to prove Theorems~\ref{thm:introtriangsurface}, \ref{thm:introtriofS3inP} and~\ref{thm:introneighbourlytriinP} involve a new approach to studying how the homotopy type of a polyhedral product is affected by the removal of a maximal face. This is likely to be of wider use. It is inspired by how certain simply-connected manifolds (not necessarily moment-angle manifolds) have their loop spaces retracting off the loops of the associated punctured manifold~\cite{T}, and makes use of key properties proved in~\cite{St1,St2} that generate retractions of looped polyhedral products with respect to $\mathcal{P}$. 

The authors would like to thank the referees for numerous helpful comments which helped to improve the paper. In particular, it was pointed out that some of our results hold for more general neighbourly triangulations of spheres.

\section{Preliminary material}
\label{sec:prelimcomb}  
This section collects some preliminary combinatorial and homotopy theoretic information. 
\medskip 

\noindent 
\textbf{Pseudomanifolds}.  
A simplicial complex $K$ of dimension $n$ is called \textit{pure} if every simplex is contained in at least one $n$-simplex. To any pure simplicial complex $K$ of dimension $n$, there is an associated graph $D(K)$ called the \textit{dual graph} of $K$. The vertices of $D(K)$ are given by the $n$-simplices of $K$, and two vertices in $D(K)$ are adjacent if and only if their respective faces in $K$ intersect over a face of codimension one. 

A simplicial complex $K$ of dimension $n$ is called a \textit{weak pseudomanifold with boundary} if every face of codimension one is contained in either one or two maximal faces. The \textit{boundary} of a weak pseudomanifold $K$ is the simplicial complex whose maximal faces are given by the codimension one faces of $K$ which are contained in exactly one maximal face. If the boundary is empty then $K$ is a \textit{weak pseudomanifold}. 

A simplicial complex $K$ is a \textit{pseudomanifold of dimension~$n$} if: (i) it is a pure simplicial complex of dimension $n$, (ii) it is a weak pseudomanifold, and (iii) $D(K)$ is a connected graph.  
The definition of a \textit{pseudomanifold with boundary} is analogous. Triangulations of manifolds are examples of pseudomanifolds. 
\medskip 

\noindent  
\textbf{Two combinatorial statements}. 
We first describe a general graph theoretic result. For a graph~$G$, let $V(G)$ be the vertex set of $G$. For a vertex $v$ in a graph $G$, the \textit{degree} of $v$, denoted $\mathrm{deg}_G(v)$, is the number of edges incident to $v$.

\begin{lemma}
\label{lem:vertexremoval}
    Let $G$ be a connected simple graph on $m$ vertices, and let $n \in \mathbb{N}$. Suppose $\mathrm{deg}_G(v) \leq n$ for all $v \in V(G)$ and there exists a vertex $w \in V(G)$ such that $\mathrm{deg}_G(w) < n$. Then there exists an ordering of the vertices $v_1,\cdots,v_m$ such that $\mathrm{deg}_G(v_1) < n$ and $\mathrm{deg}_{G \setminus \{v_1,\cdots,v_{i-1}\}}(v_i) < n$ for $2 \leq i \leq m$.
\end{lemma}
\begin{proof}
    The proof is by induction on the number of vertices. Suppose $|V(G)| \leq n$. The maximum degree of a vertex in such a graph is $n-1$, and so any ordering of the vertices suffices in this case.

    Suppose that $|V(G)| = k>n$ and the result is true for all connected graphs $H$ with $|V(H)| < k$. Let $v_1$ be a vertex of $G$ such that $\mathrm{deg}_G(v_1) < n$. Since $G$ is connected, $\mathrm{deg}_G(v_1) \geq 1$. Consider $G \setminus v_1$. By hypothesis, $\mathrm{deg}_{G}(v)\leq n$ for each vertex $v \in G \setminus v_1$, so $\mathrm{deg}_{G\setminus v_{1}}(v)\leq n$. Moreover, since $\mathrm{deg}_G(v_1) \geq 1$, there exists a vertex $v \in G \setminus v_1$ which is adjacent to $v_1$ in $G$. It follows that $\mathrm{deg}_{G \setminus v_1}(v) < n$. 
    There are two cases to consider. If $G \setminus v_1$ is connected, then the inductive hypothesis implies there is an ordering of the vertices $v_2,\cdots,v_m$ of $G \setminus v_1$ such that the statement holds. Therefore, the ordering $v_1,\cdots,v_m$ implies the result is true for $G$.

    If $G \setminus v_1$ is disconnected, denote by $C_1,\cdots,C_l$ the connected components of $G \setminus v_1$, and let $C_i$ have~$d_i$ vertices. Let $x \in C_i$ and $y \in C_j$, where $i \neq j$. Since $G$ is connected and $G \setminus v_1$ is disconnected, any path between $x$ and $y$ must pass through $v_1$. It follows that in $G$, for each $1 \leq i \leq l$, the vertex $v_1$ must be adjacent to some vertex $c_i \in C_i$, implying that $\mathrm{deg}_{G \setminus v_1}(c_i) < n$. Therefore, each $C_i$ is a connected graph with strictly less vertices than $G$ which satisfies the hypotheses in the statement of the lemma. The inductive hypothesis implies that there is an ordering of the vertices $c_{i_1},\cdots,c_{i_{d_i}}$ of $C_i$ such that the result holds for $C_i$. The ordering of the vertices $v_1,c_{1_1},\cdots,c_{1_{d_1}},\cdots,c_{l_1},\cdots,c_{l_{d_k}}$ therefore implies the result is true for $G$.
\end{proof}

Next, we describe a pushout decomposition for certain  
simplicial complexes $K$. Let $\sigma\in K$ be a maximal face 
and let $\partial\sigma$ be the boundary of $\sigma$. Let $K\backslash\sigma$ be  the simplicial complex $K$ with 
the interior of the face $\sigma$ removed.

\begin{lemma}
\label{lem:Lexists}
    Let $K$ be a simplicial complex and let $\sigma$ be a maximal face of $K$. There exists a subcomplex $L$ of $K$ such that there is a pushout 
    \[\begin{tikzcd}
	{\partial \sigma \cap L} & {\partial \sigma} \\
	L & {K \setminus \sigma}
	\arrow[from=1-1, to=1-2]
	\arrow[from=1-1, to=2-1]
	\arrow[from=1-2, to=2-2]
	\arrow[from=2-1, to=2-2]
\end{tikzcd}\] with $\partial \sigma \cap L \neq \partial \sigma$ if and only if there exists a face $\tau \in \partial \sigma$ with $|\tau| = |\sigma|-1$ which is maximal in $K \setminus \sigma$. Moreover, $L$ can be chosen to be $K \setminus \{\sigma,\tau\}$ when it exists.
\end{lemma}

\begin{proof}
        Suppose that $L$ exists but all maximal faces $\tau$ with respect to $\partial \sigma$ are not maximal in $K \setminus \sigma$. This implies there is a face $\gamma_\tau \in K \setminus \sigma$ such that $\tau \subset \gamma_\tau$. Since $\tau$ is maximal with respect to~$\partial \sigma$, the pushout implies that $\gamma_\tau$ must be contained in $L$, which in turn implies that $\tau \in L$. This is true for all maximal faces $\tau\in\partial\sigma$, so $\partial \sigma \cap L = \partial \sigma$, which is a contradiction.

    Conversely, let $\tau$ be a maximal face with respect to $\partial \sigma$ which is also maximal with respect to $K \setminus \sigma$. Define $L = K \setminus \{\sigma, \tau\}$. By definition, $\partial \sigma\cap L= \partial \sigma \setminus \tau \neq \partial \sigma$. Now let $\gamma \in K \setminus \sigma$ be such that $\gamma \notin \partial \sigma$. Since $\tau$ is maximal in $K \setminus \sigma$, we have $\tau \notin \gamma$, implying that $\gamma \in L$.
\end{proof} 

\begin{remark} 
\label{Lremark} 
It is worth noting for what comes in the next section that if $\sigma$ has dimension $d\geq 1$ then $K\setminus\sigma$ has the same vertex set as $K$. If $\sigma$ has dimension $d\geq 2$ then $L=K\setminus\{\sigma,\tau\}$ also has the same vertex set as $K$. 
\end{remark} 

\noindent 
\textbf{Spaces in $\mathcal{P}$}. 
Recall that $\mathcal{P}$ is the collection of $H$-spaces that are homotopy equivalent to a finite type product of spheres and loops on spheres. 
We state some properties of the collection $\mathcal{P}$ that will be needed. In~\cite[Theorem 3.10]{St1} it was shown that the property of being in $\mathcal{P}$ is preserved by retractions.

\begin{theorem}
    \label{thm:Pclosedunderret}
    Let $X \in \mathcal{P}$ and $A$ be a space which retracts off $X$. Then $A \in \mathcal{P}$. $\qqed$
\end{theorem}

One source of retractions in the context of polyhedral products come from the following result from \cite{DS}. If $K$ is a simplicial complex on the vertex set $[m]$ and $I\subseteq [m]$ then the \emph{full subcomplex} $K_{I}$ of $K$ on $I$ is the subcomplex of $K$ consisting of the faces of $K$ whose vertices are all in $I$. 

\begin{lemma}
    \label{lem:DS}
    If $K$ is a simplicial complex, and $K_I$ is a full subcomplex of $K$, then $\uxa^{K_I}$ retracts off $\uxa^K$. $\qqed$
\end{lemma}

We next describe two collections of polyhedral products that are in $\mathcal{P}$. Let $\mathcal{W}$ be the collection of topological spaces that are homotopy equivalent to a finite type wedge of spheres. The first result was proved in ~\cite[Theorem 1.1]{St1} and the second in~\cite[Corollary 6.5]{St2}.

\begin{theorem}
    \label{thm:graphinP}
    Let $K$ be a $1$-dimensional simplicial complex on $[m]$. Let $A_1,\cdots,A_m$ be spaces such that $\Sigma A_i \in \mathcal{W}$. Then $\Omega \caa^{K} \in \mathcal{P}$. $\qqed$
\end{theorem}

\begin{theorem}
    \label{thm:2diminP}
    Let $K$ be a $2$-dimensional simplicial complex on $[m]$. Let $A_1,\cdots,A_m$ be spaces such that $\Sigma A_i \in \mathcal{W}$. If $H_*(|L|)$ is torsion free for all full subcomplexes $L$ of $K$ with complete $1$-skeleton, then $\Omega \caa^{K} \in \mathcal{P}$. $\qqed$
\end{theorem}

If a space $X \in \mathcal{W}$, then the Hilton-Milnor theorem implies $\Omega X \in \mathcal{P}$. A result we will use to show that a space is in $\mathcal{W}$ is the following from \cite[Example 4C.2]{H}. 

\begin{lemma}
    \label{lem:homoldethomot}
    If $X$ is a simply connected space with cells in two consecutive dimensions and $H_*(X)$ is torsion free, then $X \in \mathcal{W}$. $\qqed$ 
\end{lemma}

Finally, in~\cite[Theorem 4.1]{St1} it was shown that $\mathcal{P}$ is closed under pushouts over full subcomplexes.

\begin{theorem}
\label{thm:pushoutofPisinP}
    Let $K$ be a simplicial complex defined as the pushout \[\begin{tikzcd}
	L & {K_1} \\
	{K_2} & K
	\arrow[from=1-1, to=1-2]
	\arrow[from=1-2, to=2-2]
	\arrow[from=1-1, to=2-1]
	\arrow[from=2-1, to=2-2]
\end{tikzcd}\] where either $L = \emptyset$ or $L$ is a proper full subcomplex of $K_1$ and $K_2$. If $\Sigma A_i \in \mathcal{W}$ for all $i$, $\Omega \caa^{K_1} \in \mathcal{P}$ and $\Omega \caa^{K_2} \in \mathcal{P}$, then $\Omega \caa^K \in \mathcal{P}$. $\qqed$
\end{theorem} 

\noindent 
\textbf{A homotopy pushout}. 
It will be important to identify the homotopy type of a certain homotopy pushout. 
For spaces $X$ and $Y$, the \textit{right half-smash} of $X$ and $Y$, denoted $X \rtimes Y$, is the quotient $(X \times Y)/(* \times Y)$. 

\begin{lemma} 
   \label{potype} 
   Suppose that there is a homotopy pushout 
\[\begin{tikzcd}
	{A \times B} & {D \times B} \\
	C & Q
	\arrow["{* \times 1}", from=1-1, to=1-2]
	\arrow["f", from=1-1, to=2-1]
	\arrow[from=1-2, to=2-2]
	\arrow[from=2-1, to=2-2]
\end{tikzcd}\]
   where the restriction of $f$ to $B$ is null homotopic. Let  
   \(f'\colon\namedright{A\rtimes B}{}{C}\) 
   be the quotient map and let~$E$ be its homotopy cofibre. Then there is a homotopy equivalence 
   $Q\simeq(D\rtimes B)\vee E$. 
\end{lemma} 

\begin{proof} 
Since the restriction of $f$ to $B$ is null homotopic, and the map $\ast\times 1$ is the identity on $B$, 
the space $B$ can be collapsed out of the diagram to give a homotopy pushout 
\[\begin{tikzcd}
	{A \rtimes B} & {D \rtimes B} \\
	C & Q.
	\arrow["{* \rtimes 1}", from=1-1, to=1-2]
	\arrow["f", from=1-1, to=2-1]
	\arrow[from=1-2, to=2-2]
	\arrow[from=2-1, to=2-2]
\end{tikzcd}\]
The map $\ast\rtimes 1$ is null homotopic. Thus the previous homotopy pushout can be expanded 
to a diagram of iterated homotopy pushouts 
\[\begin{tikzcd}
	{A \rtimes B} & {*} & {D \rtimes B} \\
	C & E & Q.
	\arrow[from=1-1, to=1-2]
	\arrow["{f'}", from=1-1, to=2-1]
	\arrow[from=1-2, to=1-3]
	\arrow[from=1-2, to=2-2]
	\arrow[from=1-3, to=2-3]
	\arrow[from=2-1, to=2-2]
	\arrow[from=2-2, to=2-3]
\end{tikzcd}\]
Here, in the left square the homotopy pushout of $f'$ and the constant map is the homotopy cofibre of~$f'$, 
which is $E$, and in the right square, the homotopy pushout can be identified as $Q$ since the outer rectangle 
is also a homotopy pushout. The right square itself now identifies $Q$ as $(D\rtimes B)\vee E$. 
\end{proof} 

\section{The effect on polyhedral products of removing certain maximal faces}
\label{sec:decompPP} 

Let $K$ be a simplicial complex on the vertex set $[m]$ and let $\sigma\in K$ be a maximal face. Let $K\backslash\sigma$ be $K$ with 
the interior of the face $\sigma$ removed. Observe that $\partial\sigma\subseteq K\backslash\sigma$ and there is a pushout of simplicial complexes 
\begin{equation}\label{facepo1}\begin{tikzcd}
    {{\partial \sigma}} & {\sigma } \\
    {K \setminus \sigma} & K.
	\arrow[from=1-1, to=1-2]
	\arrow[from=1-1, to=2-1]
	\arrow[from=1-2, to=2-2]
	\arrow[from=2-1, to=2-2]
\end{tikzcd}\end{equation} 
The inclusion 
\(\namedright{K\setminus\sigma}{}{K}\) 
induces a map of polyhedral products 
\(\namedright{\caa^{K\setminus\sigma}}{}{\caa^{K}}\). 
In this section, conditions are given for when this map has a right homotopy inverse. Moreover, $\caa^{K}$ is shown to be a wedge summand of $\caa^{K\setminus\sigma}$ and the complementary wedge summand is identified. 

Suppose that $\sigma$ has dimension $d\geq 2$ and there exists a face $\tau \in \partial \sigma$ with $|\tau| = |\sigma|-1$ which is maximal in $K \setminus \sigma$. 
Let $L=K\setminus\{\sigma,\tau\}$ and note that $\partial\sigma\cap L\neq\partial\sigma$. Combining Lemma~\ref{lem:Lexists} and~(\ref{facepo1}), there is an iterated pushout of simplicial complexes \begin{equation}\label{facepo}\begin{tikzcd}
	{\partial \sigma\cap L} & {{\partial \sigma}} & {\sigma } \\
	L & {K \setminus \sigma} & K.
	\arrow[from=1-1, to=1-2]
	\arrow[from=1-1, to=2-1]
	\arrow[from=1-2, to=1-3]
	\arrow[from=1-2, to=2-2]
	\arrow[from=1-3, to=2-3]
	\arrow[from=2-1, to=2-2]
	\arrow[from=2-2, to=2-3]
\end{tikzcd}\end{equation} 

As the dimension of $\sigma$ is at least $2$, by Remark~\ref{Lremark}, $L$, $K\setminus\sigma$ and~$K$ all have the same vertex set. If $\sigma\neq K$ then~$\sigma$ has a smaller vertex set than $K$, 
and we regard both $\sigma$ and $\partial\sigma$ as simplicial complexes on the vertex set $[m]$, 
giving ghost vertices which we denote by $1\leq i\leq m$ with $i\notin\sigma$. By~\cite[Proposition~3.1]{GT}, the 
iterated pushout of simplicial complexes in ~(\ref{facepo}) implies that there is an iterated pushout of polyhedral products \begin{equation}\label{GTpo}\begin{tikzcd}
	{\caa^{\partial \sigma \cap L} \times \prod\limits_{i \notin \sigma} A_i} & {\caa^{\partial \sigma} \times \prod\limits_{i \notin \sigma} A_i} & {\caa^{\sigma } \times \prod\limits_{i \notin \sigma} A_i} \\
	{\caa^L} & {\caa^{K \setminus \sigma}} & {\caa^K}
	\arrow["{i' \times 1}", from=1-1, to=1-2]
	\arrow["f", from=1-1, to=2-1]
	\arrow["{i \times 1}", from=1-2, to=1-3]
	\arrow[from=1-2, to=2-2]
	\arrow[from=1-3, to=2-3]
	\arrow[from=2-1, to=2-2]
	\arrow[from=2-2, to=2-3]
\end{tikzcd}\end{equation}
where $i'$ is induced by the inclusion $\partial \sigma \cap L \rightarrow \partial \sigma$, $i$ is induced by the inclusion 
\(\namedright{\partial\sigma}{}{\sigma}\) 
and $f$ is induced by the inclusion $\partial \sigma \cap L \rightarrow \partial \sigma$. We first show that $i'$ is null homotopic.

\begin{lemma}
\label{lem:Lnullhtpc}
    Let $\sigma$ be a maximal face of $K$ of dimension $d \geq 2$. Suppose there exists a face $\tau \in \partial \sigma$ with $|\tau| = |\sigma|-1$ which is maximal in $K \setminus \sigma$ and let $L=K\setminus\{\sigma,\tau\}$.
    Then the map of polyhedral products $\caa^{\partial\sigma\cap L} \stackrel{i'}{\rightarrow} \caa^{\partial \sigma}$ induced by the inclusion $\partial\sigma\cap L \rightarrow \partial \sigma$ is null homotopic.
\end{lemma}

\begin{proof}
        By definition of $L$, $\partial \sigma \cap L = \partial \sigma \setminus \tau$. Therefore, we show that the map of polyhedral products \[\caa^{\partial \sigma \setminus \tau} \rightarrow \caa^{\partial \sigma}\] induced by $\partial \sigma \setminus \tau \rightarrow \partial \sigma$ is null homotopic. 
    
    Let $v$ be the vertex of $\partial \sigma$ not contained in $\tau$. Then $\partial \sigma \setminus \tau = v * \partial \tau$. By definition of the polyhedral product, there are homotopy equivalences, \[\caa^{\partial \sigma \setminus \tau} \cong \caa^{\partial \tau} \times CA_v \simeq \caa^{\partial \tau},\] and the map $\caa^{\partial \sigma \setminus \tau} \rightarrow \caa^{\partial \sigma}$, up to these homotopy equivalences, becomes the map induced by the inclusion $\partial \tau \rightarrow \partial \sigma$. However, $\tau \in \partial \sigma$, and so this map factors as $\partial \tau \rightarrow \tau \rightarrow \partial \sigma$. By definition, $\caa^{\tau}$ is contractible, and so the map induced by $\partial \tau \rightarrow \partial \sigma$ is null homotopic.
\end{proof}

This allows us to give a decomposition of $\caa^{K \setminus \sigma}$ in terms of $\caa^{K}$.
 
\begin{theorem} 
   \label{faceinert} 
Let $K$ be a simplicial complex and $\sigma$ be a maximal face of $K$. Suppose there exists a face $\tau \in \partial \sigma$ with $|\tau| = |\sigma|-1$ which is maximal in $K \setminus \sigma$. Then the map 
   \(\namedright{\caa^{K\backslash\sigma}}{}{\caa^{K}}\) 
   has a right homotopy inverse and there is a homotopy equivalence 
   \[\caa^{K\backslash\sigma}\simeq\bigg(\caa^{\partial\sigma}\rtimes\prod_{i\notin\sigma} A_{i}\bigg)\vee\caa^{K}.\] 
\end{theorem} 

\begin{proof} 
From the left square of ~\eqref{GTpo}, there 
is a pushout of polyhedral products 
\[\begin{tikzcd}\label{Lpo}
	{{\caa^{\partial \sigma \cap L} \times \prod\limits_{i\notin\sigma} A_i}} & {\caa^{\partial\sigma} \times \prod\limits_{i\notin\sigma} A_i} \\
	{\caa^{L}} & {\caa^{K\setminus \sigma}.}
	\arrow["{i' \times 1}", from=1-1, to=1-2]
	\arrow["f", from=1-1, to=2-1]
	\arrow[from=1-2, to=2-2]
	\arrow[from=2-1, to=2-2]
\end{tikzcd}\]
Since $\partial \sigma \cap L \neq \partial \sigma$, Lemma ~\ref{lem:Lnullhtpc} implies that $i'$ is null homotopic. Since $L$ and $K\setminus\sigma$ have the same vertex set, $L$ has no ghost vertices. Therefore, by~\cite[Proposition~3.4]{GT}, the restriction of $f$ to 
$\prod_{i\notin\sigma} A_{i}$ is null homotopic. Thus Lemma~\ref{potype} implies 
that there is a homotopy equivalence 
\begin{equation} 
  \label{Eequiv} 
  \caa^{K\backslash\sigma}\simeq\bigg(\caa^{\partial\sigma}\rtimes\prod_{j\notin\sigma} A_{j}\bigg)\vee E 
\end{equation}  
where $E$ is the homotopy cofibre of 
\(f'\colon\namedright{\caa^{\partial\sigma\cap L}\rtimes\prod_{i\notin\sigma} A_{i}}{}{\caa^{L}}\). 

The next step is to identify $E$. By ~\eqref{GTpo}, there is an iterated diagram of pushouts of polyhedral products 
\[\begin{tikzcd}
	{\caa^{\partial \sigma \cap L} \times \prod\limits_{i \notin \sigma} A_i} & {\caa^{\partial \sigma} \times \prod\limits_{i \notin \sigma} A_i} & {\caa^{\sigma } \times \prod\limits_{i \notin \sigma} A_i} \\
	{\caa^L} & {\caa^{K \setminus \sigma}} & {\caa^K.}
	\arrow["{i' \times 1}", from=1-1, to=1-2]
	\arrow["f", from=1-1, to=2-1]
	\arrow["{i \times 1}", from=1-2, to=1-3]
	\arrow[from=1-2, to=2-2]
	\arrow[from=1-3, to=2-3]
	\arrow[from=2-1, to=2-2]
	\arrow[from=2-2, to=2-3]
\end{tikzcd}\] Since $\caa^{\sigma}$ is contractible, this diagram of iterated pushouts is equivalent up to homotopy 
to the iterated diagram of homotopy pushouts 
\[\begin{tikzcd}
	{\caa^{\partial \sigma \cap L} \times \prod\limits_{i \notin \sigma} A_i} & {\caa^{\partial \sigma} \times \prod\limits_{i \notin \sigma} A_i} & {\prod\limits_{i \notin \sigma} A_i} \\
	{\caa^L} & {\caa^{K \setminus \sigma}} & {\caa^K}
	\arrow["{{i' \times 1}}", from=1-1, to=1-2]
	\arrow["f", from=1-1, to=2-1]
	\arrow["{{\pi_2}}", from=1-2, to=1-3]
	\arrow[from=1-2, to=2-2]
	\arrow[from=1-3, to=2-3]
	\arrow[from=2-1, to=2-2]
	\arrow[from=2-2, to=2-3]
\end{tikzcd}\]
where $\pi_{2}$ is the projection. As noted above, the restriction of $f$ to $\prod_{i\notin\sigma} A_{i}$ 
is null homotopic, so all the vertical maps restrict trivially to $\prod_{i\notin\sigma} A_{i}$, implying that this factor may be collapsed out to give an iterated diagram of homotopy pushouts 
\begin{equation}\label{collapseA} 
\begin{tikzcd}
	{\caa^{\partial \sigma \cap L} \rtimes \prod\limits_{i \notin \sigma} A_i} & {\caa^{\partial \sigma} \rtimes \prod\limits_{i \notin \sigma} A_i} & {*} \\
	{\caa^L} & {\caa^{K \setminus \sigma}} & {\caa^K}
	\arrow["{{i' \rtimes 1}}", from=1-1, to=1-2]
	\arrow["{f'}", from=1-1, to=2-1]
	\arrow[from=1-2, to=1-3]
	\arrow[from=1-2, to=2-2]
	\arrow[from=1-3, to=2-3]
	\arrow[from=2-1, to=2-2]
	\arrow[from=2-2, to=2-3]
\end{tikzcd} 
\end{equation}
In particular, all three vertical maps have the same homotopy cofibre. By definition, the homotopy 
cofibre of $f'$ is $E$, while the right vertical map clearly has $\caa^{K}$ as its homotopy cofibre. 
Thus $E\simeq\caa^{K}$ and therefore from~(\ref{Eequiv}) there is a homotopy equivalence 
\[\caa^{K\backslash\sigma}\simeq 
     \bigg(\caa^{\partial\sigma}\rtimes\prod_{j\notin\sigma} A_{j}\bigg)\vee\caa^{K}.\] 
Further, the right homotopy inverse for 
\(\namedright{\caa^{K \setminus \sigma}}{}{E}\), 
together with the bottom row of~(\ref{collapseA}), gives a composite 
\(\nameddright{E}{}{\caa^{K\backslash\sigma}}{}{\caa^{K}}\) 
that is a homotopy equivalence. Thus the map 
\(\namedright{\caa^{K\backslash\sigma}}{}{\caa^{K}}\) 
has a right homotopy inverse. 
\end{proof} 

\section{Polyhedral products associated to pseudomanifolds with boundary}
\label{sec:pseudowithbound}

In order to study polyhedral products associated to pseudomanifolds, we first consider the case with non-trivial boundary. The results will be used frequently in subsequent sections when the boundary is empty. For a simplicial complex $K$, and an integer $t \geq 0$, let $K^t$ be the $t$-skeleton of $K$. If $K$ has dimension $n$, we apply the results from the previous section in order to show that $\caa^K$ retracts off $\caa^{K^{n-1}}$ under certain hypotheses. 

\begin{theorem}
    \label{thm:maniwithboundretskel}
    Let $K$ be an $n$-dimensional, pure, weak pseudomanifold with boundary having $\ell$ maximal faces $\sigma_1,\cdots,\sigma_\ell$. Suppose that each connected component of $D(K)$ contains a vertex of degree strictly less than $n+1$. Then there is a homotopy equivalence \[\caa^{K^{n-1}} \simeq \bigvee\limits_{i=1}^\ell \left(\caa^{\partial \sigma_{i}} \rtimes \prod\limits_{j \notin \sigma_{i}}A_j\right) \vee \caa^K\] and the map of polyhedral products $\caa^{K^{n-1}} \rightarrow \caa^{K}$ induced by the inclusion $K^{n-1} \rightarrow K$ has a right homotopy inverse.
\end{theorem}
\begin{proof}
    Applying Lemma ~\ref{lem:vertexremoval} to each connected component of $D(K)$ and relabelling the maximal faces if necessary, we can assume $\sigma_1$ has degree strictly less than $n+1$ in $D(K)$, and for $2 \leq i \leq \ell$, $\sigma_i$ has degree strictly less than $n+1$ in $D(K) \setminus \{\sigma_1,\cdots,\sigma_{i-1}\}$. Define $K_0 = K$, and for $1 \leq i \leq \ell$, define $K_i = K\setminus \{\sigma_1,\cdots,\sigma_i\}$. Observe that by definition, $K_\ell = K^{n-1}$. There is a sequence of inclusions \[K^{n-1} = K_{\ell} \rightarrow K_{\ell-1} \rightarrow \cdots \rightarrow K_1 \rightarrow K_0 = K,\] which factors the inclusion of $K^{n-1}$ into $K$.

    We show that for each $i$, the map of polyhedral products $\caa^{K_i} \rightarrow \caa^{K_{i-1}}$ induced by the inclusion $K_i \rightarrow K_{i-1}$ has a right homotopy inverse. Since $\sigma_i$ has degree strictly less than $n+1$ in $D(K) \setminus \{\sigma_1,\cdots,\sigma_{i-1}\}$, there exists a face $\tau \in \partial \sigma_i$ with $|\tau| = |\sigma_i|-1$ which is contained in only one maximal face, namely $\sigma_i$. In particular, $\tau$ is maximal in $K_{i} = K_{i-1} \setminus \sigma_i$. Hence, Theorem ~\ref{faceinert} implies there is a homotopy equivalence \[\caa^{K_{i}} \simeq \left(\caa^{\partial \sigma_i}\rtimes \prod\limits_{j \notin \sigma_i} A_j\right) \vee \caa^{K_{i-1}},\] and the map of polyhedral products $\caa^{K_{i}} \rightarrow \caa^{K_{i-1}}$ induced by the inclusion $K_{i} \rightarrow K_{i-1}$ has a right homotopy inverse. The homotopy equivalence asserted by the theorem then follows by induction.
\end{proof}

Theorem ~\ref{thm:maniwithboundretskel} has a homological consequence that will be important in Section ~\ref{sec:loopMaM}.

\begin{proposition}
\label{prop:torfreeret}
    Let $K$ be an $n$-dimensional, pure, weak pseudomanifold with boundary. Suppose that each connected component of $D(K)$ contains a vertex of degree strictly less than $n+1$. Let $A_1,\cdots,A_m$ be spaces such that $H_*(A_i)$ is torsion free for all $i$. Then $H_*(\caa^K)$ is torsion free if and only if $H_*(\caa^{K^{n-1}})$ is torsion free. 
\end{proposition}
\begin{proof}
    If $H_*(\caa^{K^{n-1}})$ is torsion free then, by Theorem ~\ref{thm:maniwithboundretskel}, $\caa^{K}$ retracts off $\caa^{K^{n-1}}$, implying that $H_*(\caa^{K})$ is torsion free.

    Now suppose that $H_*(\caa^{K})$ is torsion free. By Theorem ~\ref{thm:maniwithboundretskel}, there is a homotopy equivalence \[\caa^{K^{n-1}} \simeq \bigvee\limits_{i=1}^\ell \left(\caa^{\partial \sigma_{i}} \rtimes \prod\limits_{j \notin \sigma_{i}}A_j\right) \vee \caa^K,\] where $\sigma_1,\ldots,\sigma_{\ell}$ are maximal faces of $K$. By assumption, $H_{\ast}(\caa^{K})$ is torsion free, so to show that $H_{\ast}(\caa^{K^{n-1}})$ is torsion free it suffices to show that $H_*(\caa^{\partial \sigma_{i}} \rtimes \prod_{j \notin \sigma_{\ell}}A_j)$ is torsion free for $1 \leq i \leq \ell$. Let $\sigma_i = \{j_1,\cdots,j_n\}$. By \cite[Theorem 1.7]{IK3} or \cite[Theorem 1.1]{GT}, there is a homotopy equivalence \[\caa^{\partial \sigma_{i}} \simeq\Sigma^{n-1} A_{j_1}\wedge \cdots\wedge A_{j_n}.\] In particular, $\caa^{\partial \sigma_{i}}$ is a suspension. In general, if $A$ is a suspension then there is a homotopy equivalence $A\rtimes B\simeq A\vee (A\wedge B)$, so in our case there is a homotopy equivalence \[\caa^{\partial \sigma_{i}} \rtimes \prod\limits_{j \notin \sigma_{i}}A_j \simeq \caa^{\partial \sigma_{i}} \vee (\caa^{\partial \sigma_{i}} \wedge \prod_{j \notin \sigma_{i}}A_j).\] 
    By hypothesis, each $H_{\ast}(A_{i})$ is torsion free, so the reduced K\"unneth theorem implies that both $H_{\ast}(\caa^{\partial \sigma_{i}})$ and $H_{\ast}(\caa^{\partial \sigma_{i}} \wedge \prod\limits_{j \notin \sigma_{i}}A_j)$ are torsion free, and hence $H_{\ast}(\caa^{\partial \sigma_{i}} \rtimes \prod\limits_{j \notin \sigma_{i}}A_j)$ is torsion free.
\end{proof} 

Theorem ~\ref{thm:maniwithboundretskel} can also be used to give coarse decompositions of the loop spaces of polyhedral products associated to pseudomanifolds with boundary in low dimensions.

\begin{theorem}
    \label{thm:pseudowithbound2and3dim}
    Let $K$ be a pure, weak pseudomanifold with boundary of dimension $n$ on $[m]$, and let $A_1,\cdots,A_m$ be spaces such that $\Sigma A_i \in \mathcal{W}$ for all $i$. Suppose that each connected component of $D(K)$ contains a vertex of degree strictly less than $n+1$. If $n=1$, $n=2$, or $n=3$ and $H_*(|L|)$ is torsion free for all subcomplexes $L$ of $K$ with complete $1$-skeleton, then $\Omega \caa^K \in \mathcal{P}$.
\end{theorem}
\begin{proof}
    If $n=1$, then Theorem ~\ref{thm:graphinP} implies $\Omega \caa^K \in \mathcal{P}$, so assume $n \geq 2$. By Theorem ~\ref{thm:maniwithboundretskel}, $\caa^K$ retracts off $\caa^{K^{n-1}}$, and so $\Omega \caa^K$ retracts off $\Omega \caa^{K^{n-1}}$. By Theorem ~\ref{thm:Pclosedunderret}, to show $\Omega \caa^K \in \mathcal{P}$ it suffices to show that $\Omega \caa^{K^{n-1}} \in \mathcal{P}$. But Theorem ~\ref{thm:graphinP} when $n=2$ and Theorem ~\ref{thm:2diminP} when $n=3$ imply that $\Omega \caa^{K^{n-1}} \in \mathcal{P}$.
\end{proof}

\section{Polyhedral products associated to pseudomanifolds}
\label{sec:pseudo}

The results from the previous section are applied to certain classes of pseudomanifolds. In particular, we show that loop spaces of certain polyhedral products associated to surfaces are in $\mathcal{P}$. We start with a general statement giving conditions for when a polyhedral product has its loop space in $\mathcal{P}$. 

\begin{theorem}
\label{thm:restinPimplyP}
    Let $K$ be a simplicial complex on $[m]$ that does not have a complete $1$-skeleton. Let $A_1,\cdots,A_m$ be spaces such that $\Sigma A_i \in \mathcal{W}$ for all $i\in [m]$. If $\Omega \caa^{K \setminus i} \in \mathcal{P}$ for all $i \in [m]$ then $\Omega \caa^K \in \mathcal{P}$.
\end{theorem}
\begin{proof} 
For a vertex $v \in K$, let $N(v)$ be the set of vertices adjacent to $v$ in the $1$-skeleton of $K$.  
Since~$K$ does not have a complete $1$-skeleton, there exists a vertex $v$ such that $v \cup N(V) \neq K^0$. By \cite[Lemma 4.4]{St1}, there is a pushout of simplicial complexes \[\begin{tikzcd}
	{K_{N(v)}} & {K_{v \cup N(v)}} \\
	{K \setminus v} & K.
	\arrow[from=1-1, to=1-2]
	\arrow[from=1-1, to=2-1]
	\arrow[from=1-2, to=2-2]
	\arrow[from=2-1, to=2-2]
\end{tikzcd}\] 
If $\Omega \caa^{K \setminus v} \in \mathcal{P}$ and $\Omega \caa^{K_{v \cup N(v)}} \in \mathcal{P}$ then Theorem ~\ref{thm:pushoutofPisinP} implies that $\Omega \caa^K \in \mathcal{P}$. 

By assumption, $\Omega \caa^{K \setminus v} \in \mathcal{P}$. 
For $\Omega \caa^{K_{v \cup N(v)}}$, since $v \cup N(v) \neq K^0$, there exists a vertex $w$ such that $v \cup N(v)$ is a full subcomplex of $K \setminus w$. By Lemma ~\ref{lem:DS}, $\caa^{v \cup N(v)}$ retracts off $\caa^{K \setminus w}$, and so $\Omega \caa^{v \cup N(v)}$ retracts off $\Omega \caa^{K \setminus w}$. By assumption, $\Omega \caa^{K \setminus w} \in \mathcal{P}$, and so Theorem ~\ref{thm:Pclosedunderret} implies that $\Omega \caa^{v \cup N(v)} \in \mathcal{P}$. 
\end{proof}

Theorem ~\ref{thm:restinPimplyP} will be used to show that low dimensional pseudomanifolds which do not have a complete $1$-skeleton have their associated polyhedral products in $\mathcal{P}$. To do this, we first show that if $K$ is a pseudomanifold, then $K \setminus i$ satisfies the hypotheses of Theorem ~\ref{thm:maniwithboundretskel}.

\begin{lemma}
\label{lem:restsathypo}
    Let $K$ be a pseudomanifold of dimension $n$ on $[m]$. For any $i \in [m]$, $K \setminus i$ is a pure simplicial complex of dimension $n$, a weak pseudomanifold with boundary, and each connected component of $D(K \setminus i)$ contains a vertex of degree strictly less than $n+1$.
\end{lemma}
\begin{proof}
    First, we show that $K \setminus i$ is pure of dimension $n$. Suppose $\sigma$ is a maximal face of $K \setminus i$ of dimension $k < n$. By assumption, $K$ is pure of dimension $n$ so $\sigma$ must be contained in some maximal simplex $\sigma'\in K$ with $i \in \sigma'$. Since $i\notin\sigma$, there must exist a codimension one face $\tau \subset \sigma'$ such that $\sigma \subseteq \tau$ and $i \notin \tau$. Moreover $K$ is a pseudomanifold, and so in $K$, $\tau$ is contained in two maximal faces, $\sigma'$ and $\sigma''$. However, since~$\sigma'$ contains $i$ and $\tau$ is of codimension one, $\sigma''$ does not contain $i$, and therefore $\sigma'' \in K \setminus i$. Since $\sigma \subseteq \tau$, this implies $\sigma \subset \sigma''$, which is a contradiction. Thus every maximal face of $K\setminus i$ has dimension $n$, implying that $K\setminus i$ is pure of dimension $n$.

    Next, we show that $K \setminus i$ is a weak pseudomanifold with boundary. Let $\tau$ be a codimension one face of $K \setminus i$. In $K$, since $K$ is a pseudomanifold, $\tau$ is contained in two maximal faces, $\sigma$ and $\sigma'$. At most one of $\sigma$ and $\sigma'$ contains the vertex $i$, otherwise $\sigma=\tau\cup\{i\}=\sigma'$. Therefore one of $\sigma$ and $\sigma'$ is in $K \setminus i$. Hence, $\tau$ is contained in either one or two maximal faces in $K \setminus i$. We now show that the boundary of $K\setminus i$ is non-empty. Since $K$ is pure, the vertex $i$ must be contained in at least one maximal face $\sigma''$ in $K$. Hence, if $\tau'$ is the codimension one face of $\sigma''$ which does not contain $i$, then it follows that $\tau'$ is contained in the boundary of $K \setminus i$. 

    Finally, we show that each connected component of $D(K \setminus i)$ contains a vertex of degree strictly less than $n+1$. Since $K$ is a pseudomanifold of dimension $n$, each maximal face contains $n+1$ codimension one faces, each of which is contained in two distinct maximal faces. Therefore, each vertex in $D(K)$ has degree $n+1$. The graph $D(K \setminus i)$ is obtained from $D(K)$ by removing vertices corresponding to maximal faces of $K$ containing the vertex $i$. If $D(K\setminus i)$ is connected, then since $D(K)$ is connected, there must exist a vertex in $D(K\setminus i)$ which is adjacent to at least one of the vertices removed from $D(K)$. Therefore, there must exist a vertex in $D(K\setminus i)$ with degree strictly less than $n+1$. Now suppose $D(K\setminus i)$ is disconnected, and let $x,y \in D(K \setminus i)$ be two vertices in different connected components. Since $D(K)$ is connected and $D(K \setminus i)$ is disconnected, any path in $D(K)$ between $x$ and $y$ must pass through one of the vertices removed from $D(K)$ to obtain $D(K \setminus i)$. Therefore, for each connected component of $D(K\setminus i)$, there must exist a vertex $v$ such that $v$ is adjacent to at least one of the vertices removed from $D(K)$. Hence, $v$ has degree strictly less than $n+1$ in $D(K\setminus i)$.
\end{proof}

\begin{theorem}
\label{thm:23dimpseudoinP}
    Let $K$ be either a $2$-dimensional pseudomanifold or a $3$-dimensional pseudomanifold such that $H_*(|L|)$ is torsion free for all full subcomplexes $L$ of $K$ with complete $1$-skeleton. Suppose that $K$ is on the vertex set $[m]$ and $A_1,\cdots,A_m$ are spaces such that $\Sigma A_i \in \mathcal{W}$ for all $i\in [m]$. If $K$ does not have complete $1$-skeleton then $\Omega \caa^K \in \mathcal{P}$.
\end{theorem}
\begin{proof}
    For all $i \in [m]$, Lemma ~\ref{lem:restsathypo} implies $K \setminus i$ satisfies the hypotheses of Theorem ~\ref{thm:pseudowithbound2and3dim}. Therefore, $\Omega \caa^{K \setminus i} \in \mathcal{P}$ for all $i \in [m]$. Since $K$ does not have a complete $1$-skeleton, Theorem~\ref{thm:restinPimplyP} implies that $\Omega \caa^K \in \mathcal{P}$.
\end{proof}

A special case of pseudomanifolds of dimension $2$ are connected, orientable, closed surfaces. In this case, we can give a complete picture of $\Omega \caa^K$ without the assumption on the $1$-skeleton. 

\begin{theorem}
\label{thm:triangsurface}
    Let $K$ be the triangulation of a connected, orientable, closed surface on $[m]$. Let $A_1,\cdots,A_m$ be spaces such that $\Sigma A_i \in \mathcal{W}$. Then $\Omega \caa^K \in \mathcal{P}$.
\end{theorem}
\begin{proof} 
Since $K$ is the triangulation of a connected, orientable, closed surface, for each $I \subseteq [m]$, $|K_I|$ embeds into $\mathbb{R}^3$. By \cite[Corollary 3.46]{H}, this implies that $H_*(|K_I|)$ is torsion free. Therefore, Theorem ~\ref{thm:2diminP} implies that $\Omega \caa^K \in \mathcal{P}$.
\end{proof}

A special case of Theorem ~\ref{thm:triangsurface} proves Theorem~\ref{thm:introtriangsurface}. 

\begin{proof}[Proof of Theorem~\ref{thm:introtriangsurface}] 
Take each pair $(CA_{i},A_{i})$ in Theorem~\ref{thm:triangsurface} to be $(D^{2},S^{1})$. 
\end{proof} 

\section{Loop space decompositions of moment-angle manifolds}
\label{sec:loopMaM}

In this section, we specialise to moment-angle complexes associated to triangulations of spheres, all of which are pseudomanifolds. If $K$ is a triangulation of $S^{2}$ then $\Omega\zk\in\mathcal{P}$ by Theorem ~\ref{thm:introtriangsurface}. We will prove an analogous result if $K$ is a triangulation of $S^3$. To start, we consider more general properties of a family of odd dimensional sphere triangulations called neighbourly triangulations. Let $K$ be a triangulation of $S^n$ on $[m]$. In this case, $\zk$ has the structure of a manifold of dimension $m+n+1$ \cite[Theorem 4.1.4]{BP} which is $2$-connected. 
\medskip 

\noindent 
\textbf{Pseudomanifolds and the minimally non-Golod property}. 
An important algebraic condition on simplicial complexes is the notion of Golodness. A simplicial complex $K$ on $[m]$ is called \emph{Golod} if all cup products and higher Massey products in $H^*(\zk)$ are trivial, and $K$ is \emph{minimally non-Golod} if $K\setminus i$ is Golod for all $i \in [m]$. For example, if $\zk$ is a suspension, or a co-$H$-space, then all cup products and higher Massey products vanish in $H^*(\zk)$, implying that $K$ is Golod.

We focus our attention on a special family of odd dimensional sphere triangulations. Recall from the Introduction that a simplicial complex $K$ is called $k$\textit{-neighbourly} if every set of $k+1$ vertices spans a simplex. A triangulation~$K$ of a sphere $S^{2n+1}$ is called \emph{neighbourly} if $K$ is $n$-neighbourly. It was shown in \cite[Proposition~3.6]{L} that if $K$ is the boundary of a dual polytope and neighbourly, then $K$ is minimally non-Golod. Gitler and Lopez de Medrano \cite[Theorem 1.3]{GLdM} showed that in this case the corresponding $\zk$ is diffeomorphic to a connected sum of sphere products, with two spheres in each product. We give an anaologue of Limonchenko's result that holds for any $n$-neighbourly $(2n+1)$-dimensional pseudomanifold. This requires a suspension splitting of moment-angle complexes from \cite[Corollary 2.23]{BBCG}, known as the BBCG decomposition.

\begin{theorem}
    \label{thm:BBCG}
    Let $K$ be a simplicial complex. There is a homotopy equivalence \[\Sigma \zk \simeq \bigvee\limits_{I \notin K}\Sigma^{2+|I|}|K_I|\] that is natural for inclusions of simplicial complexes. $\qqed$
\end{theorem}

The BBCG decomposition for $\Sigma \zk$ ``desuspends" if there is a homotopy equivalence 
\[\zk \simeq \bigvee\limits_{I \notin K}\Sigma^{1+|I|}|K_I|.\]
Observe that if the BBCG decomposition desuspends then $\zk$ is a suspension, and so is Golod.

\begin{theorem}
    \label{thm:neighbourlupseudominnonGolod}
    Let $K$ be a pseudomanifold on $[m]$ of dimension $2n+1$. If $K$ is $n$-neighbourly then the BBCG decomposition for $\Sigma \mathcal{Z}_{K \setminus i}$ desuspends for all $i \in [m]$. Consequently, $K$ is either Golod or minimally non-Golod.
\end{theorem}
\begin{proof}
    By \cite[Theorem 1.3]{IK2}, for any simplicial complex $K$, $\zk$ is a co-$H$ space if and only if the BBCG decomposition desuspends. Hence, it suffices to show that $\mathcal{Z}_{K \setminus i}$ is a co-$H$ space for all $i \in [m]$. Since $K$ is a pseudomanifold, Lemma ~\ref{lem:restsathypo} implies $K \setminus i$ satisfies the hypotheses of Theorem ~\ref{thm:maniwithboundretskel}, implying that $\mathcal{Z}_{K \setminus i}$ retracts off $\mathcal{Z}_{(K \setminus i)^{2n}}$. The simplicial complex $(K \setminus i)^{2n}$ is an $n$-neighbourly, $2n$-dimensional simplicial complex, so by \cite[Theorem 1.6]{IK2}, the BBCG decomposition for $\Sigma \mathcal{Z}_{(K \setminus i)^{2n}}$ desuspends. Thus $\mathcal{Z}_{(K\setminus i)^{2n}}$ is a suspension. As $\mathcal{Z}_{K \setminus i}$ retracts off $\mathcal{Z}_{(K \setminus i)^{2n}}$, $\mathcal{Z}_{K \setminus i}$ is therefore a co-$H$ space. 
\end{proof}

If $K$ is a triangulation of $S^n$, we can characterise when $K$ is Golod. If $K = \partial \Delta^{n+1}$, then Theorem ~\ref{thm:BBCG} implies that $\zk$ has one non-trivial homology group, and therefore has no nontrivial cup products or Massey products, implying that $K$ is Golod. If $K \neq \partial \Delta^{n+1}$, then Theorem ~\ref{thm:BBCG} implies that a minimal missing face corresponds to a $\mathbb{Z}$ summand in $H^i(\zk)$, where $i < m+n+1$. If $x\in H^{i}(\zk)$ generated this summand, then as $\zk$ is a manifold, Poincar\'e duality implies there is a class $y\in H^{m+n+1-i}(\zk)$ such that $x\cup y\neq 0$. Thus $H^{\ast}(\zk)$ has non-trivial cup products, implying that $K$ is not Golod. Therefore, we obtain the following. \begin{lemma}
\label{lem:GolodMaM}
    If $K$ is a triangulation of $S^n$ then $K$ is Golod if and only if $K = \partial \Delta^{n+1}$. $\qqed$ 
\end{lemma}

\noindent 
\textbf{Neighbourly trianglulations of $S^{2n+1}$}. To start, let $K$ be a triangulation of $S^{n}$ on $[m]$.
Let $\overline{\zk}$ be the $(m+n)$-skeleton of $\zk$. There is a homotopy cofibration \[S^{n+m} \stackrel{f}{\rightarrow} \overline{\zk} \rightarrow \zk\] 
where $f$ attaches the $(m+n+1)$-cell to $\zk$. We aim for a decomposition of $\overline{\zk}$ under certain hypotheses. These hypotheses will be satisfied when $K$ is a neighbourly triangulation of an odd dimensional sphere. First, we determine the homology of $\overline{\zk}$.

\begin{proposition}
    \label{prop:hmlyZkandskel}
    Let $K$ be a triangulation of $S^n$ on $[m]$. There are isomorphisms \[H_*(\zk) \cong \bigoplus\limits_{I \notin K} H_*(\Sigma^{1+|I|} |K_I|)\qquad H_*(\overline{\zk}) \cong \bigoplus\limits_{I \notin K,I \neq [m]} H_*(\Sigma^{1+|I|} |K_I|).\]
\end{proposition}
\begin{proof}
The first isomorphism follows from Theorem ~\ref{thm:BBCG}. For the second, one summand has been deleted, corresponding to $I=[m]$. When $I=[m]$ then $K_{I}=K$. Since $K$ is a triangulation of a sphere, $|K| = S^n$, so $\Sigma^{1+\vert [m]\vert}\vert K\vert\simeq S^{m+n+1}$. This accounts for the generator in $H_{m+n+1}(\zk)$. As $\overline{\zk}$ is the $(m+n)$-skeleton of $\zk$, the second isomorphism follows.
\end{proof}

 In case the BBCG decomposition for $\Sigma \mathcal{Z}_{K \setminus i}$ desuspends for each $i \in [m]$ we can decompose $\overline{\zk}$.

\begin{proposition}
    \label{prop:decompskelzk}
    Let $K$ be a triangulation of $S^{n}$ on $[m]$. If the BBCG decomposition for $\Sigma \mathcal{Z}_{K \setminus i}$ desuspends for all $i \in [m]$, then $K$ is Golod when $K = \partial \Delta^{n-1}$ or minimally non-Golod when $K  \neq \partial \Delta^{n-1}$, and there is a homotopy equivalence \[\overline{\zk} \simeq \bigvee \limits_{I \notin K,I \neq [m]} \Sigma^{1+|I|} |K_I|.\]
\end{proposition}
\begin{proof} 
    The BBCG decomposition for $\Sigma\zk$ is 
    \[\Sigma\mathcal{Z}_{K} \simeq \bigvee\limits_{I \notin K} \Sigma^{2+|I|} |K_I|.\] 
    Consider the map 
    \(\namedright{\mathcal{Z}_{K\setminus i}}{}{\zk}\) 
    induced by the inclusion 
    \(\namedright{K\setminus i}{}{K}\).
    The naturality of the BBCG decomposition implies that the decomposition of $\Sigma\mathcal{Z}_{K\setminus i}$ may be obtained by restricting the decomposition for $\Sigma\zk$ to those full subcomplexes $K_{I}$ with $I\notin K$ and $i\notin I$. As the BBCG decomposition for $\Sigma \mathcal{Z}_{K\setminus i}$ desuspends by hypothesis, we obtain a homotopy equivalence 
    \[\mathcal{Z}_{K\setminus i} \simeq \bigvee\limits_{I \notin K,i \notin I} \Sigma^{1+|I|} |K_I|.\]
    Taking the wedge sum of the inclusion maps 
    \(\namedright{\mathcal{Z}_{K\setminus i}}{}{\zk}\) 
    over all $i\in[m]$ then gives a map 
    \[\namedright{\bigvee_{i=1}^{m}\bigg(\bigvee\limits_{I \notin K,i \notin I} \Sigma^{1+|I|} |K_I|\bigg)}{}{\zk}.\]
    Observe that the index set on the left accounts for all $I\notin K$ except for an instance of $I$ that contains each $i\in [m]$, of which there is only one, $I=[m]$. However, the index set may include multiple copies of the same wedge summand. Restricting to a single copy for each instance of $I\notin K$, $I\neq [m]$, we obtain a map 
    \[g\colon\namedright{\bigvee\limits_{I \notin K,I\neq [m]} \Sigma^{1+|I|} |K_I|}{}{\zk}\]
    whose suspension induces the inclusion of all wedge summands in the BBCG decomposition of $\zk$ except for the $I=[m]$ summand. In particular, $g$ induces an injection in homology. As each wedge summand $\Sigma^{1+\vert I\vert}\vert K_{I}\vert$ has dimension~$<m+n+1$ for $I\neq [m]$, the map $g$ factors through the $(m+n)$-skeleton $\overline{\mathcal{Z}_{K}}$ of $\zk$ to give a map  
    \[g'\colon\namedright{\bigvee\limits_{I \notin K,I\neq [m]} \Sigma^{1+|I|} |K_I|}{}{\overline{\mathcal{Z}_{K}}}.\]    
    Since $g$ induces an injection in homology, so does $g'$. The description of $H_{\ast}(\overline{\mathcal{Z}_{K}})$ in Proposition~\ref{prop:hmlyZkandskel} therefore implies that $g'$ must induce an isomorphism in homology, and hence $g'$ is a homotopy equivalence by Whitehead's Theorem. 
   \end{proof}

We will show that Proposition ~\ref{prop:decompskelzk} holds when $K$ is a neighbourly triangulation of $S^{2n+1}$. In this case, the decomposition of $\overline{\zk}$ can be refined. The following argument is essentially due to Gitler and Lopez de Medrano \cite{GLdM}, and the authors thank a referee for pointing out the following result holds for all neighbourly triangulations of $S^{2n+1}$, rather than just $S^3$.

\begin{theorem}
    \label{thm:neighbourlytriofsphere}
    If $K$ is a neighbourly triangulation of $S^{2n+1}$ on $[m]$ with $n \geq 1$ then the simplicial complex $K$ is Golod when $K = \partial \Delta^{2n+2}$, or minimally non-Golod when $K \neq \partial \Delta^{2n+2}$. Moreover, $\overline{\zk} \in \mathcal{W}$.
\end{theorem} 
\begin{proof}
Consider the real moment-angle complex $\mathbb{R}\zk := (D^1,S^0)^K$ associated to $K$, which is a closed topological manifold of dimension $2n+2$ \cite[Theorem 4.1.7]{BP}. By \cite[Corollary 2.24]{BBCG}, there is a homotopy equivalence \begin{equation}\label{eqn:RMAChom}\Sigma \mathbb{R}\zk \simeq \bigvee\limits_{I \notin K} \Sigma^2 |K_I|.\end{equation} 

Since $K$ is a neighbourly triangulation of $S^{2n+1}$, each full subcomplex $K_I$ has $H_k(|K_I|) = 0$ for all $k < n$. It follows from \eqref{eqn:RMAChom} that $\mathbb{R}\zk$ is $n$-connected. By Poincar\'e duality, the reduced homology of $\mathbb{R}\zk$ is non-trivial only in degrees $n+1$ and $2n+2$. Therefore since $H_{2n+2}(\mathbb{R}\zk) \cong \mathbb{Z}$ and $K = S^{2n+1}$, \eqref{eqn:RMAChom} implies that for all $I \subseteq[m]$ with $I \neq [m]$, $\widetilde{H}_k(|K_I|)$ can be non-trivial if and only if $k=n$. Hence either $\vert K_{I}\vert$ is contractible or homotopy equivalent to $S^{n}$. In particular, each such $\Sigma |K_I| \in \mathcal{W}$.

Combining Theorem ~\ref{thm:neighbourlupseudominnonGolod}, Lemma ~\ref{lem:GolodMaM}, and Proposition ~\ref{prop:decompskelzk}, we then obtain the desired result.
\end{proof} 

Now we can prove Theorem~\ref{thm:introneighbourlytriinP}, which states 
that if $K$ is a neighbourly triangulation of $S^{2n+1}$ then $\Omega\zk\in P$. 

\begin{proof}[Proof of Theorem ~\ref{thm:introneighbourlytriinP}] 
Theorem ~\ref{thm:neighbourlytriofsphere} implies that $\overline{\zk} \in \mathcal{W}$. The Hilton-Milnor theorem then implies that $\Omega \overline{\zk}\in \mathcal{P}$. Using the fact that $\zk$ is a manifold, by \cite[Example 5.4]{T}, the inclusion $\overline{\zk} \rightarrow \zk$ has a right homotopy inverse after looping. Hence, Theorem ~\ref{thm:Pclosedunderret} implies that $\Omega \zk \in \mathcal{P}$.
\end{proof}

\noindent 
\textbf{Triangulations of $S^{3}$}.  
Now we specialise to any triangulation $K$ of $S^{3}$ and prove Theorem~\ref{thm:introtriofS3inP}, which states that $\Omega\zk\in\mathcal{P}$. This splits into two cases, the first where $K$ has a complete $1$-skeleton, and the second where it does not. The first case follows from Theorem~\ref{thm:introneighbourlytriinP} and the second 
requires a preliminary homological result from \cite[Lemma 3.4.12]{Si} on the homology of $\zk$. We provide a proof for completeness.

 \begin{lemma}
     \label{lem:S3trihomology}
     Let $K$ be a triangulation of $S^3$ on $[m]$. Then $H_*(\zk)$ is torsion free. $\qqed$
 \end{lemma}
 \begin{proof}
Since $K$ is a triangulation of $S^3$, $H_*(|K|)$ is torsion free. If $I \subseteq [m]$ with $I \neq [m]$, $|K_I|$ embeds into $S^3 \setminus \{pt\} \cong \mathbb{R}^3$, and so by \cite[Corollary 3.46]{H}, $H_*(|K_I|)$ is torsion free. Therefore, Proposition ~\ref{prop:hmlyZkandskel} implies that $H_*(\zk)$ is torsion free.
 \end{proof}

We can now prove Theorem~\ref{thm:introtriofS3inP}. 

\begin{proof}[Proof of Theorem~\ref{thm:introtriofS3inP}.]
    By Lemma ~\ref{lem:S3trihomology}, $H_*(\zk)$ is torsion free, so Theorem ~\ref{thm:BBCG} implies that $H_*(|K_I|)$ is torsion free for all $I \subseteq [m]$. If the $1$-skeleton of $K$ is not a complete graph, then Theorem ~\ref{thm:23dimpseudoinP} implies that $\Omega \zk \in \mathcal{P}$.

   If the $1$-skeleton is a complete graph, then Theorem ~\ref{thm:introneighbourlytriinP} implies $\Omega \zk \in \mathcal{P}$.
\end{proof}

\begin{remark}By a result of Cai \cite[Corollary 2.10]{C}, $\zk$ is a manifold if and only if $K$ is a generalised homology sphere. It would be interesting to know if these results also hold when $K$ is a generalised homology sphere, but not a triangulation of a sphere.\end{remark} 

\begin{remark}
    Not every triangulation $K$ of a sphere will result in $\Omega\zk\in P$. For example, let $L$ be the $6$-vertex triangulation of $\mathbb{R}P^2$. By \cite[Example 3.3]{GPTW}, there is a homotopy equivalence \begin{equation}\label{eqn:6vertRP2}\mathcal{Z}_L \simeq W \vee \Sigma^7 \mathbb{R}P^2,\end{equation} where $W \in \mathcal{W}$. As in \cite[Theorem 3.2]{LW}, one can construct a triangulation of $S^4$ containing $L$ as a full subcomplex by applying certain stellar subdivisions to $\partial\Delta^{5}$. Let $K$ be such a triangulation. By Lemma ~\ref{lem:DS} and \eqref{eqn:6vertRP2}, $\Sigma^7 \mathbb{R}P^2$ retracts off $\zk$, and so $\Omega \Sigma^7 \mathbb{R}P^2$ retracts off $\Omega \zk$. This implies that $H_*(\Omega \zk)$ contains $2$-torsion and so $\Omega \zk \notin \mathcal{P}$.
\end{remark}

\noindent 
\textbf{Quasitoric manifolds}. 
Theorem~\ref{thm:introtriofS3inP} will be applied in Proposition~\ref{quasitoric} to show similar results for certain manifolds known as quasitoric manifolds. As in~\cite{DJ}, a $2n$-dimensional manifold has a 
\emph{locally standard} $T^{n}$-action if locally it is the standard action of $T^{n}$ on $\mathbb{C}^{n}$. 
A \emph{quasitoric manifold} over an $n$-dimensional simple polytope $P$ is a closed, smooth 
$2n$-dimensional manifold $M$ that has a smooth locally standard 
$T^{n}$-action for which the orbit space $M/T^{n}$ is homeomorphic to~$P$ as a manifold with corners. 

Let $P$ be an $n$-dimensional simple polytope with $m$ facets, and let $K = \partial P^*$ be the Alexander dual of the boundary of $P$. The simplicial complex $K$ is a triangulation of $S^{n-1}$, and therefore $\mathcal{Z}_K$ is a moment-angle manifold. By \cite[Proposition 7.3.12]{BP}, a quasitoric manifold $M$ of dimension~$2n$ arises as a quotient $M \cong \zk/T^{m-n}$ for some subtorus $T^{m-n} \subseteq T^{m}$ that acts freely on the corresponding moment-angle complex $\zk$. The quotient description of $M$ implies that there is a principal $T^{m-n}$-fibration 
\begin{equation} 
  \label{qtfib} 
  \nameddright{T^{m-n}}{}{\mathcal{Z}_{K}}{}{M}. 
\end{equation} 
The following lemma is well known to experts in the area.

\begin{lemma} 
   \label{loopM} 
   Let $M$ be a quasi-toric manifold of dimension $2n$ associated to a polytope $P$ of dimension $n$. Let $K = \partial P^*$. Then there is a homotopy equivalence 
   $\Omega M\simeq T^{m-n}\times\Omega\mathcal{Z}_{K}$. 
\end{lemma} 
\begin{proof} 
Consider the homotopy fibration 
\(\nameddright{\Omega M}{r}{T^{m-n}}{}{\mathcal{Z}_{K}}\) 
induced by~(\ref{qtfib}). By~\cite[Proposition 4.3.5 (a)]{BP}, $\mathcal{Z}_{K}$ is $2$-connected. Therefore $r$ induces an isomorphism on $\pi_{1}$. Each $\mathbb{Z}$ generator of $\pi_{1}(\Omega M)$ is the Hurewicz image of a map 
\(\namedright{S^{1}}{}{\Omega M}\), 
and the loop space structure allows these to be multiplied together to obtain a map 
\(s\colon\namedright{T^{m-n}}{}{\Omega M}\). 
The composite $r\circ s$ therefore induces an isomorphism on $\pi_{1}$. As $T^{m-n}$ is an Eilenberg-Mac Lane space, this implies $r\circ s$ is a homotopy equivalence. Thus $\Omega M\simeq T^{m-n}\times\Omega\zk$.  
\end{proof}

\begin{proposition} 
    \label{quasitoric} 
    If $M$ is a quasitoric manifold of dimension $4$, $6$ or $8$, then $\Omega M\in\mathcal{P}$.    
\end{proposition} 

\begin{proof} 
If $M$ is $2n$-dimensional with $m$ facets then, by Lemma~\ref{loopM}, there is a homotopy equivalence 
$\Omega M\simeq T^{m-n}\times \Omega\mathcal{Z}_{K}$, where $K$ is the Alexander dual of the boundary of an $n$-dimensional polytope. To show that $\Omega M \in \mathcal{P}$, it therefore suffices to show that $\Omega \zk \in \mathcal{P}$. But the hypotheses that $M$ has dimension $4$, $6$ or $8$ implies that $K$ is a triangulation of $S^1$, $S^2$ or $S^3$ respectively. Theorem ~\ref{thm:graphinP} in the first case, Theorem ~\ref{thm:introtriangsurface} in the second case, and Theorem ~\ref{thm:introtriofS3inP} in the third case imply that $\Omega \zk \in \mathcal{P}$, as required.
\end{proof}


\end{document}